\documentclass[a4paper,11pt,twoside]{amsart}
\usepackage[english]{babel}
\usepackage[utf8]{inputenc}

\usepackage[a4paper,inner=3cm,outer=3cm,top=4cm,bottom=4cm,pdftex]{geometry}
\usepackage{fancyhdr}
\usepackage{titlesec}
\usepackage{enumerate}
\usepackage{color}
\usepackage{bold-extra}
\titleformat{\section}{\normalfont\scshape\centering}{\thesection}{1em}{}
\titleformat{\subsection}{\bfseries}{\thesubsection}{1em}{}
  
\usepackage{comment}
\usepackage{graphics}
\usepackage{aliascnt}
\usepackage[pdftex,citecolor=green,linkcolor=red]{hyperref}

\usepackage{amsmath}
\usepackage{amsfonts}
\usepackage{amssymb}
\usepackage{amsthm}
\usepackage{mathtools}

\newtheorem{theorem}{Theorem}[section]

\newtheorem{lemma}[theorem]{Lemma}

\theoremstyle{definition}

\numberwithin{equation}{section}

\newcommand{\ord}{\textup{ord}}

\author{Olli J\"arviniemi}
\address{Department of Mathematics and Statistics, P.O. Box 68, 00014 Helsinki, Finland}
\email{olli.jarviniemi@helsinki.fi}

\title{Positive lower density for prime divisors of generic linear recurrences}

\date{}

\begin{document}
\begin{abstract}
Let $d \ge 3$ be an integer and let $P \in \mathbb{Z}[x]$ be a polynomial of degree $d$ whose Galois group is $S_d$. Let $(a_n)$ be a linearly recuresive sequence of integers which has $P$ as its characteristic polynomial. We prove, under the generalized Riemann hypothesis, that the lower density of the set of primes which divide at least one element of the sequence $(a_n)$ is positive.
\end{abstract}

\maketitle

\section{Introduction}
\label{sec:intro}

Given a sequence of integers, it is natural to consider the set of primes which divide at least one of its values (the prime divisors of the sequence). Here we consider the prime divisors of linearly recursive sequences. We assume that the elements $a_1, a_2, \ldots$ of the sequence and the coefficients $c_i$ in the defining recursion
$$a_{n+d} + c_{d-1}a_{n+d-1} + \ldots + c_0a_n = 0$$
are integers. The order of $(a_n)$ is (the minimum possible value of) $d$.

It is a well-known result often attributed to Pólya \cite{polya} that, excluding degenerate cases, a linearly recursive sequence has infinitely many prime divisors. One naturally wonders: how dense is the set of prime divisors of a linearly recursive sequence (with respect to the set of primes)?

The case of second order sequences has been studied in a number of works under the assumption of the generalized Riemann hypothesis (GRH), under which it is known that the density of prime divisors exists, is positive  (unless there are only finitely many prime divisors) and can, at least in principle, be computed explicitly (see \cite{ballot}, \cite{hasse}, \cite{lagarias}, \cite{moree}, \cite[Section 8.4]{moreeSurvey}, \cite{moree-stevenhagen2}, \cite{moree-stevenhagen}, \cite{moree-stevenhagen3}, \cite{stephens}, \cite{stevenhagen}). Unconditional results are much more modest: in \cite{mss} it is proved that the number of primes $p \le x$ which are prime divisors of such a sequence is at least of magnitude $\log x$.

Higher order sequences have received considerably less attention. In \cite{roskamPD} Roskam states that ``Essentially nothing is known for sequences of order larger than $2$'', and it is mentioned that some very non-generic cases can be handled (see \cite{ballot}). Roskam proceeds to proving that, under a certain generalization of Artin's conjecture on primitive roots, ``generic'' linear recurrences have a positive lower density of prime divisors. However, we note that a much stronger result under a significantly weaker assumption follows directly from the work of Niederreiter \cite{niederreiter}. Nevertheless, while the regular version of Artin's conjecture has been proven by Hooley \cite{hooley} under GRH (see \cite{moree} for a survey), even the weakened version of the generalization seems to be out of reach under standard conjectures. See Section \ref{sec:discussion} for more discussion.

Here we prove, under GRH, that the set of prime divisors of a ``generic'' linear recurrences has positive lower density. This seems to be the first such result which is applicable to ``almost all'' sequences and which assumes only standard conjectures.

\begin{theorem}
\label{thm:main}
Assume GRH. Let $d \ge 3$ be an integer and let $P \in \mathbb{Z}[x]$ be a polynomial whose Galois group is the symmetric group $S_d$ and such that the quotient of no two roots of $P$ is a root of unity. Let $(a_n)$ be a linearly recursive sequence of integers whose characteristic polynomial is $P$. The set of primes which divide at least one element of the sequence $(a_n)$ has a lower density of at least $1/(d-1)$. In particular, this lower density is strictly positive. 
\end{theorem}
The assumption on the quotients of the roots of $P$ (the \emph{non-degeneracy}) follows from the assumption on the Galois group for $d \ge 4$.

Note that almost all polynomials of degree $d$ have Galois group isomorphic to $S_d$, so the result applies to ``$100\%$'' of linear recurrences. The proof actually works for a slightly wider class of recurrences, and proves that almost all primes $p$ such that $P$ has suitable factorization modulo $p$ are prime divisors of the sequence.

\begin{theorem}
\label{thm:general}
Assume GRH. Let $P \in \mathbb{Z}[x]$ be a polynomial which has the following properties.
\begin{itemize}
\item $d := \deg(P) \ge 3$.
\item $P$ is irreducible.
\item There are infinitely many primes $p$ such that $P$ factorizes as the product of a linear polynomial and an irreducible polynomial of degree $d-1$ modulo $p$, that is, the Galois group of $P$ contains an element whose cycle type is $(1, d - 1)$.
\item If $|P(0)| > 1$, the roots of $P$ are multiplicatively independent, and if $|P(0)| = 1$, some (or, equivalently, any) $d-1$ roots of $P$ are multiplicatively independent.
\end{itemize}
Let $(a_n)$ be a linearly recursive sequence of integers whose characteristic polynomial is $P$. Then almost all primes $p$ such that $P$ factorizes as the product of irreducibles of degree $1$ and $d-1$ modulo $p$ divide at least one element of the sequence $(a_n)$. In particular, the set of primes which divide at least one element of the sequence $(a_n)$ has a strictly positive lower density.
\end{theorem}
(There are non-trivial examples of polynomials $P$ which do not satisfy the last condition \cite{drmota-skalba}.)

The proof also adapts to reducible characteristic polynomials in certain special cases, but as such cases are rare, we do not discuss them in detail.

The GRH used in the proofs is that the non-trivial zeros of the Dedekind zeta-function of any number field lie on the line $\text{Re}(s) = 1/2$. (For details, see Lemma \ref{lem:artinIdeal} below and \cite[Theorem 3.1]{lenstra}.)

We first provide a proof sketch, after which we give a detailed argument. We conclude by discussing challenges arising in the study of prime divisors of linear recurrences.

\section{Overview of the method}
\label{sec:overview}

For concreteness we consider the sequence $a_n$ defined by
$$a_n = 5^n + (3 + \sqrt{2})^n + (3 - \sqrt{2})^n, n = 1, 2, \ldots$$
The characteristic polynomial $(x-5)(x - (3 + \sqrt{2}))(x - (3 - \sqrt{2}))$ is reducible and thus not of the form of Theorem \ref{thm:main}, but we only use this example to demonstrate the idea. (The proof of Theorem \ref{thm:main} can be adapted to this sequence, though.)

In the case when $2$ is a quadratic residue modulo $p$ the period of the sequence $(a_n)$ modulo $p$ divides $p-1$. We are unable to say anything about whether such primes are prime divisors of $(a_n)$ or not. 

The case when $2$ is a quadratic nonresidue modulo $p$, however, turns out to be accessible. Write $n = (p+1)k + r, k, r \in \mathbb{Z}$. We may view the numbers $3 \pm \sqrt{2}$ as elements of $\mathbb{F}_{p^2}$, and by norms $(3 \pm \sqrt{2})^{p+1} = 7$. Hence
\begin{align*}
a_{(p+1)k + r} \equiv 5^{2k+r} + 7^k\left((3 + \sqrt{2})^r + (3 - \sqrt{2})^r\right) \pmod{p}.
\end{align*}
The equation $a_{(p+1)k + r} \equiv 0 \pmod{p}$ may thus be written as
\begin{align}
\label{equ:1}
\left(\frac{5^2}{7}\right)^k \equiv -\left(\frac{3 + \sqrt{2}}{5}\right)^r - \left(\frac{3 - \sqrt{2}}{5}\right)^r \pmod{p}.
\end{align}

Artin's primitive root conjecture states that a given rational number $a$ is a primitive root modulo $p$ for infinitely many primes $p$ as long as $a$ is not $-1$ or a square. Under GRH one can prove this in a quantative form: the set of such primes has a positive density (as long as its infinite) \cite{hooley}. This density is often quite large. For example, for $a = 2$ the density is roughly $37\%$. 

It turns out that the order of a rational number $a$ modulo primes is almost always almost maximal assuming $a \not\in \{-1, 0, 1\}$ (under GRH). More precisely, the density of primes $p$ with $\ord_p(a) \ge (p-1)/C$ goes to $1$ as $C \to \infty$. (See \cite[Section 5]{wagstaff} or Lemma \ref{lem:artinIdeal} below.)

Note also that if $\ord_p(a) = (p-1)/h$, then the function $x \to a^x \pmod{p}, x \in \mathbb{Z}$ attains all non-zero $h$th powers modulo $p$ as its values.

In this light, to prove that \eqref{equ:1} is solvable for almost any prime $p$ it suffices to show that the equation
\begin{align}
\label{equ:2}
x^h \equiv -\left(\frac{3 + \sqrt{2}}{5}\right)^r - \left(\frac{3 - \sqrt{2}}{5}\right)^r \pmod{p}
\end{align}
has a solution $(x, r)$ with $x \neq 0$ for almost any prime $p$.

Note the right hand side is a linear recurrence. Denote it by $b_r$. We aim to prove that the sequence $b_1, b_2, \ldots$ almost always attains a $h$th power as its value modulo $p$.

By using results from Galois theory and the Chebotarev density theorem, it is not very hard to show that this is true, for example, if there is an infinite subsequence of $(b_r)$ whose elements are distinct primes.

Of course, we cannot guarantee that a linear recurrence has infinitely many prime values. However, there are only a very few cases where such an idea does not work. To name one, if $b_r$ is always two times a square, then if $2$ is a quadratic nonresidue modulo $p$ (which happens for a positive density of primes, and in this example case for all $p$ we consider), the sequence $b_r$ may avoid all squares modulo $p$.

In general, the only obstructions are that the values of the linear recurrence are almost perfect powers. By applying Zannier's result on Pisot's $d$th root conjecture \cite{zannier} we reduce to considering whether a linearly recursive sequence arising in the proof is the power of another recurrence. From here only elementary observations are needed.

As we already mentioned, the sequence $a_n$ considered here is not of the form Theorem \ref{thm:main}, and the general case is more complicated. There are two notable differences.

To perform the ``norm-trick'' and to arrive to an equation of the form \eqref{equ:1} we need to control the norms of the roots of the characteristic polynomial in finite fields. To do so, in the situation of Theorem \ref{thm:main} we consider those primes $p$ for which $P$ factorizes as the product of irreducibles of degrees $1$ and $d-1$ modulo $p$. 

From here we are able to reduce to an equation similar to \eqref{equ:2}, though this time the right hand side is not necessarily a linear recurrence of integers but of algebraic numbers. By taking norms we reduce to the integer case, the same idea can be implemented and we are able to show that the reduction of at least one term of the sequence of algebraic numbers to $\mathbb{F}_p$ is a $h$th power almost always.

In Section \ref{sec:order} we state the GRH-conditonal result mentioned earlier. In Section \ref{sec:red} we reduce the problem to a polynomial equation in a similar manner as above. We note that the equation is solvable for a set of primes of density $1$ if certain field extensions are linearly disjoint. We present the tool to handle such questions in Section \ref{sec:galois}. To apply the tool we have to check that our sequence is not roughly a perfect power, as done in Section \ref{sec:perfect}. We wrap up the proof in Section \ref{sec:conclude}.

\section{Orders of reductions of algebraic numbers}
\label{sec:order}

The following lemma is used when transforming our problem into a polynomial equation. This lemma is the only part of the proof that relies on GRH.

\begin{lemma}
\label{lem:artin}
Assume GRH. Let $P \in \mathbb{Z}[x]$ be non-constant and irreducible. Assume $P$ is not a cyclotomic polynomial, i.e. no root of $P$ is a root of unity, and that $P$ is not the identity. Let $S$ denote the set of primes such that the equation $P(x) \equiv 0 \pmod{p}$ has at least one solution $f(p)$. For $C > 0$ let $S_C$ denote the set of primes $p \in S$ for which the order of $f(p)$ modulo $p$ is at least $(p-1)/C$. The (lower) density of $S_C$ with respect to $S$ approaches $1$ as $C \to \infty$.
\end{lemma}

Note that we do not say anything about which root $f(p)$ of $P$ modulo $p$ we choose if there are several of them. The result is true no matter how the choices are made.

This result is equivalent to the following algebraic number theoretic formulation. 

\begin{lemma}
\label{lem:artinIdeal}
Assume GRH. Let $\alpha$ be a non-zero algebraic number which is not a root of unity. Let $K = \mathbb{Q}(\alpha)$ and let $O_K$ denote the ring of integers of $K$. Let $T$ denote the set of prime ideals of $O_K$ whose norm is a prime. For $C > 0$ let $T_C$ denote the set of primes $\mathfrak{p}$ of $T$ such that the reduction of $\alpha$ in $O_K/\mathfrak{p} \cong \mathbb{F}_p$ has order at least $(p-1)/C$, where $p$ is the norm of $\mathfrak{p}$. The (lower) density of $T_C$ with respect to $T$ approaches $1$ as $C \to \infty$ (where ideals are ordered by norm).
\end{lemma}

\begin{proof} (Cf. \cite[Section 5]{wagstaff}.) Note that almost all ideals of $O_K$ belong to $T$. For $k \in \mathbb{Z}_+$ let $T_k' = T_k \setminus T_{k-1}$. The results of Lenstra \cite{lenstra} imply that $T_k'$ has a density for all $k$. This density is given by
$$d(T_k') = \sum_{t = 1}^{\infty} \frac{\mu(t)}{[K(\zeta_{kt}, \alpha^{1/kt}) : K]},$$
where the sum is absolutely convergent. Let $f(n) = 1/[K(\zeta_n, \alpha^{1/n}) : K]$. Now rearranging gives
$$\sum_{k = 1}^{\infty} \sum_{t = 1}^{\infty} \mu(t)f(kt) = \sum_{K = 1}^{\infty} f(K) \sum_{d \mid K}  \mu(d) = f(1) = 1.$$
\end{proof}

\section{Reduction to a polynomial equation}
\label{sec:red}

We consider the setup of Theorem \ref{thm:main}. The proof also works in the situation of Theorem \ref{thm:general}.

Let $P$ and $(a_n)$ be as in Theorem \ref{thm:main}. Assume $(a_n)$ is never zero. Let $S$ denote the set of primes $p$ such that $P$ factorizes as the product of polynomials of degree $1$ and $d-1$ modulo $p$. By the assumption and the Chebotarev density theorem the density of $S$ is $1/(d-1)$. We prove that the density of primes of $S$ which are prime divisors of $(a_n)$ is one.

We write
$$a_n = \gamma_1\alpha_1^n + \ldots + \gamma_d\alpha_d^n,$$
where $\alpha_1, \ldots , \alpha_n$ are the roots of $P$ and $\gamma_1, \ldots , \gamma_d$ are constants, and (for further purposes) we define
$$b_n = b_{h, n} := \alpha_1^{hn}\left(-\frac{a_n}{\gamma_1\alpha_1^n} + 1\right)$$
and
$$c_n = c_{h, n} := N_{F/\mathbb{Q}}(b_n) = (-1)^dN^{hn}\prod_{i = 1}^d \sum_{j \neq i} \frac{\gamma_j\alpha_j^n}{\gamma_i\alpha_i^n},$$
where $F$ is the splitting field of $P$, $N := N_{F/\mathbb{Q}}(\alpha_1)$ and $h \in \mathbb{Z}_+$ is a parameter fixed later. For an algebraic number field $K$ we let $O_K$ denote the ring of integers of $K$.

Note that $\gamma_i \in \mathbb{Q}(\alpha_i) \setminus \{0\}$ (see e.g. \cite[Section 2]{roskamPD}). 

For each (unramified) prime $p \in S$ there exists a homomorphism $\varphi : O_F \to \mathbb{F}_{p^{d-1}}$ mapping one of the roots $\alpha_i$, say $\alpha_1$, to an element of $\mathbb{F}_p$, and the other roots to elements of $\mathbb{F}_{p^{d-1}}$ whose degrees over $\mathbb{F}_p$ are $d-1$. Let $K = \mathbb{Q}(\alpha_1)$.

Note that for $n = k(p^{d-1} - 1)/(p-1) + r$ one has
$$\varphi(\alpha_1)^n = \varphi(\alpha_1)^{kd + r},$$
$$\varphi(\alpha_i)^n = N_{\mathbb{F}_{p^{d-1}}/\mathbb{F}_p}(\varphi(\alpha_i))^k\varphi(\alpha_i)^r, 2 \le i \le d$$
and
$$N_{\mathbb{F}_{p^{d-1}}/\mathbb{F}_p}(\alpha_i) = \varphi(\alpha_2) \cdots \varphi(\alpha_d) = \frac{N}{\varphi(\alpha_1)}, 2 \le i \le d$$ 
(Clearly $\varphi(\alpha_1) \neq 0$ for all but finitely many $p$.) Hence
\begin{align*}
\varphi(a_n) = \varphi(\gamma_1)\varphi(\alpha_1)^{kd + r} + \left(\frac{N}{\varphi(\alpha_1)}\right)^k\left(\varphi(\gamma_2)\varphi(\alpha_2)^r + \ldots + \varphi(\gamma_d)\varphi(\alpha_d)^r\right)
\end{align*}
so $\varphi(a_n) = 0$ if and only if
$$\varphi(\alpha_1^{d+1}/N)^k = -\frac{\varphi(\gamma_2)\varphi(\alpha_2)^r + \ldots + \varphi(\gamma_d)\varphi(\alpha_d)^r}{\varphi(\gamma_1)\varphi(\alpha_1)^r}.$$

We then note that $\alpha_1^{d+1}/N$ is not a root of unity: Otherwise the conjugates $\alpha_i^{d+1}/N$ are roots of unity as well. Hence the product
$$\prod_{1 \le i \le d} \alpha_i^{d+1}/N = N$$
is a root of unity. Hence $N = \pm 1$, and thus the numbers $\alpha_i$ are roots of unity. This contradicts the non-degeneracy of $P$.

We may hence apply Lemma \ref{lem:artin} to $\alpha_1^{d+1}/N$. It suffices to show that for any $h \in \mathbb{Z}_+$ the equation
$$x^h = b_r, x \ne q0$$
is solvable in $\mathbb{F}_p$ for almost all primes $p \equiv 1 \pmod{h}$. This may be reformulated as follows: almost all prime ideals of $O_{K(\zeta_h)}$ split in at least one of the fields $K_r := K(\zeta_h, b_r^{1/h}), r = 1, 2, \ldots$ 

\section{A linear disjointness result}
\label{sec:galois}

We use the following basic result from Galois theory \cite{garrett}.

\begin{lemma}
\label{lem:garrett}
Let $h$ be a positive integer, let $K$ be a number field containing a $h$th root of unity, and let $a_1, a_2, \ldots , a_k$ be a sequence of integers. Assume that for any integers $0 \le e_i < h$, not all zero, one has $a_1^{e_1/h} \cdots a_k^{e_k/h} \not\in K.$ Then
$$[K(a_1^{1/h}, \ldots , a_k^{1/h}) : K] = h^k.$$
\end{lemma}

\begin{lemma}
\label{lem:coprime}
Let $(x_n)$ be a linearly recursive sequence of integers. Assume that there do not exist a number field $K$ and integers $A, B \ge 1$, $D \ge 2$ such that $x_{An+B}$ is always a $D$th power of an element of $K$. Then there exist a subsequence $x_{n_1}, x_{n_2}, \ldots$ of $(x_n)$ and primes $p_1, q_1, p_2, q_2, \ldots$ such that
\begin{itemize}
\item $\gcd(v_{p_i}(x_{n_i}), v_{q_i}(x_{n_i})) = 1$ for all $i$
\item the primes $p_1, q_1, p_2, q_2, \ldots$ are pairwise distinct
\end{itemize}
\end{lemma}

\begin{proof}
Note that the condition implies that $(x_n)$ has infinitely many prime divisors -- otherwise choose $A = 1, B = 0$, $D = 2$ and $K$ to be $\mathbb{Q}(\sqrt{p_1}, \sqrt{p_2}, \ldots , \sqrt{p_n})$, where $p_1, \ldots , p_n$ are the prime divisors of $(x_n)$. Note also that for all except finitely many primes $p$, say for all $p$ not belonging to $T$, the sequence $(x_n)$ is periodic modulo $p^k$ for all $k \in \mathbb{Z}_+$. 

We will inductively choose the indices $n_i$ and the primes $p_i, q_i$, additionally requiring that $p_i, q_i \not\in T$. Assume we have already choosen some $n_1, \ldots , n_k, p_1, \ldots , p_k, q_1, \ldots , q_k$ satisfying the conditions. We now pick $n_{k+1}, p_{k+1}, q_{k+1}$. Let $S = \{p_1, q_1, \ldots , p_k, q_k\} \cup T$.

Pick some prime divisor $p_{k+1} \not\in S$, let $v_{p_{k+1}}(x_{n_0}) = t > 0$ for some $n_0$ such that $x_{n_0} \neq 0$. By periodicity modulo $p_{k+1}^{t+1}$, there exists an arithmetic progression $An + B, n = 1, 2, \ldots$ such that $v_{p_{k+1}}(x_{An+B}) = t$ for all $n$

If there exist some prime $q_{k+1} \not\in S \cup \{p_{k+1}\}$ and $n \in \mathbb{Z}_+$ such that $\gcd(v_{q_{k+1}}(x_{An+B}), t) = 1$, we are done. Assume this is not the case.

If for any prime $q_{k+1} \not\in S$ and any $n$ we had $\gcd(v_{q_{k+1}}(x_{An+B}), t) = t$, then we could write 
$$|x_{An+B}| = f(n)^t\prod_{s \in S} s^{f_s(n)},$$
for some functions $f, f_s : \mathbb{Z}_+ \to \mathbb{Z}_{\ge 0}$. Then $x_{An+B}$ would always be a perfect $t$th power in a number field containing the $t$th roots of all primes of $s$ and a $2t$th root of unity, contrary to the assumption.

Hence there exist a prime $q_{k+1} \not\in S$ and $n \in \mathbb{Z}_+$ such that $t' := \gcd(v_{q_{k+1}}(x_{An+B}), t) < t$. Repeat the above argument with $q_{k+1}$ in place of $p_{k+1}$ and $t'$ in place of $t$. The value of $t$ decreases. It must happen that for some value of $p_{k+1}$ and $t$ we find a prime $q_{k+1}$ with $\gcd(v_{q_{k+1}}(x_{An+B}), t) = 1$.
\end{proof}

\begin{lemma}
\label{lem:linDis}
Let $h \in \mathbb{Z}_+$ and let $F$ be a number field containing a $h$th root of unity. Let $(x_n)$ be a linearly recursive sequence of integers satisfying the assumption of Lemma \ref{lem:coprime}. Then there exists a subsequence $x_{n_1}, x_{n_2}, \ldots$ of $(x_n)$ such that the extensions
$$F(x_{n_1}^{1/h})/F, F(x_{n_2}^{1/h})/F, \ldots$$
are linearly disjoint and of degree $h$.
\end{lemma}

\begin{proof}
Note first that there exists a constant $c$ (depending on $F$) such that if $x$ is an integer with $x^{1/h} \in F$, then all prime divisors of $x$ which are greater than $c$ have multiplicity divisible by $h$. (If $p$ is a prime with $h \nmid v_p(x)$, then $p$ is ramified in the $\mathbb{Q}(x^{1/h})$. If $\mathbb{Q}(x^{1/h}) \subset F$, then $p$ is ramified in $F$, and only finitely many primes ramify in $F$.)

Let $(x_{n_i})$ be a subsequence constructed in Lemma \ref{lem:coprime}. We may assume that the corresponding primes $p_i, q_i$ are larger than $c$. We prove that this subsequence works by showing that
$$[F(x_{n_1}^{1/h}, x_{n_2}^{1/h}, \ldots , x_{n_k}^{1/h}) : F] = h^k$$
for all $k$. Apply Lemma \ref{lem:garrett}. Assume that
\begin{align}
\label{equ:4}
x_{n_1}^{e_1/h} \cdots x_{n_k}^{e_k/h} \in F, 0 \le e_i < h.
\end{align}
By the choice of $c$, \eqref{equ:4} implies that the prime divisors of $x_{n_1}^{e_1} \cdots x_{n_k}^{e_k}$ which are larger than $c$ have multiplicity divisible by $h$. By the choice of the primes $p_i, q_i$ this implies $h \mid e_i$ for all $i$.
\end{proof}

\section{Linear recurrences and perfect powers}
\label{sec:perfect}

In this section we show that Lemma \ref{lem:linDis} may be applied to the sequence $(c_n)$. Assume not, so $c_{An+B}$ is always a $D$th power of an element in a fixed number field.

By a result of Zannier \cite{zannier}, the only case when a linear recurrence is always a $D$th power is when it is the $D$th power of a linear recurrence. We may hence write
$$c_{An+B} = d_n^D,$$
where $d$ is a linearly recursive sequence (whose elements are not necessarily integers). Write $d$ as an exponential polynomial
$$d_n = \sum_{m = 0}^{t} \left(n^t \sum_{k = 1}^{u_m} e_{m, k}\sigma_{m, k}^n\right),$$
where the coefficients $e_{m, k}$ are non-zero, $\sigma_{m, 1}, \sigma_{m, 2}, \ldots , \sigma_{m, u_m}$ are non-zero and pairwise distinct, and $i_m > 0$ for all $m \le t$. 

We first show that $t = 0$, i.e. that the characteristic polynomial of $(d_n)$ has no repeated roots. Assume not. Now one sees that $d_n^D$ may be written as
$$n^{Dt} \left(\sum_{k = 1}^{u_t} e_{t, k}\sigma_{t, k}^n\right)^D + e_n,$$
where $e_n$ is an exponential polynomial whose polynomial terms have degree less than $Dt$. Since the representation of a linear recurrence as an exponential polynomial is unique, one sees that $c_{An+B} =  d_n^D$ cannot hold for all $n \in \mathbb{Z}$.

We may thus write
$$d_n = e_1\sigma_1^n + \ldots + e_u\sigma_u^n.$$
Each of $\sigma_i$ can be written in the form $\alpha_1^{x_{i, 1}}\alpha_2^{x_{i, 2}} \cdots \alpha_d^{x_{i, d}}$, where $x_{i, j}$ are rational numbers \cite{ritt}. Let $M \in \mathbb{Z}_+$ be such that $x_{i, j}M \in \mathbb{Z}$ for all $i, j$. Write the equation $c_{AMn+B} = d_{Mn}^D$ as
$$(-1)^dN^{h(AMn + B)}\prod_{i = 1}^d \sum_{j \neq i} \frac{\gamma_j\alpha_j^{B}\alpha_j^{AMn}}{\gamma_i\alpha_i^B\alpha_i^{AMn}} = \left(\sum_{i = 1}^u e_i \prod_{j = 1}^d \alpha^{x_{i, j}Mn}\right)^D$$

Note then that if $\prod_{i = 1}^d \alpha_i^{f_i} \in \mathbb{Q}$ for some integers $f_i$, then $f_i = f_j$ for all $i, j$. Indeed: Since the Galois group of $P$ is $S_d$, we have $\prod_{i = 1}^d \alpha_i^{f_i'} = \prod_{i = 1}^d \alpha_i^{f_i}$ for any permutation $f_i'$ of $f_i$. Hence $\alpha_i^{f_i}\alpha_j^{f_j} = \alpha_i^{f_j}\alpha_j^{f_i}$, so $(\alpha_i/\alpha_j)^{f_i - f_j} = 1$, which by non-degeneracy of $P$ implies $f_i = f_j$.

Hence, if $N \neq \pm 1$, one has $\prod_{i = 1}^d \alpha_i^{f_i} = 1$ only if $f_i = 0$ for all $i$. By basic results on linear recurrences, this implies that for any $Q \in \mathbb{C}[X_1^{\pm 1}, \ldots , X_d^{\pm 1}]$ in $d$ variables we have
$$Q(\alpha_1^{n}, \alpha_2^{n}, \ldots , \alpha_d^{n}) = 0$$
for all integers $n$ if and only if $Q$ is identically zero. Hence
$$(-1)^dN^{hB}(X_1 \cdots X_d)^{hAM}\prod_{i = 1}^d \sum_{j \neq i} \frac{\gamma_j\alpha_j^BX_j^{AM}}{\gamma_i\alpha_i^BX_i^{AM}} = \left(\sum_{i = 1}^u e_i\prod_{j = 1}^d X_j^{x_{i, j}M}\right)^D$$
identically as elements of $\mathbb{C}[X_1^{\pm 1}, \ldots , X_d^{\pm 1}]$.

In particular, the left hand side is a perfect $D$th power in $\mathbb{C}[X_1^{\pm 1}, \ldots , X_d^{\pm 1}]$. Perform a suitable transformation of the form $X_i \to c_iX_i$, clear out constants and simplify. One obtains that
\begin{align*}
(X_1 \cdots X_d)^{(h-1)A'}\prod_{i = 1}^d (X_1^{A'} + \ldots + X_d^{A'} - X_i^{A'})
\end{align*}
is a $D$th power of a polynomial, where $A' := AM$. This is clearly impossible if $D \nmid (h-1)A'$. Otherwise we may drop the term $(X_1 \cdots X_d)^{(h-1)A'}$. One sees that the polynomials
$$X_1^{A'} + \ldots + X_d^{A'} - X_i^{A'}$$
are pairwise coprime, and hence each of them must be a $D$th power. This is not the case, as can be seen, for example, by considering the partial derivative of $X_2^{A'} + \ldots + X_d^{A'}$ with respect to $X_2$ at $(X_2, X_3, \ldots , X_d) = (x_0, 1, \ldots , 1)$, where $x_0^{A'} + (d-2) = 0$.

The case $N = \pm 1$ is handled similarly: For any polynomial $Q \in \mathbb{C}[X_1^{\pm 1}, \ldots , X_{d-1}^{\pm 1}]$ we have
$$Q(\alpha_1^n, \alpha_2^n, \ldots , \alpha_{d-1}^n) = 0$$
for all integers $n$ only if $Q$ is zero. Proceeding as before, we have that
$$\prod_{i = 1}^d (X_1^{A'} + \ldots + X_d^{A'} - X_i^{A'})$$
is a $D$th power of a polynomial in the variables $X_1, \ldots , X_{d-1}$ (with possibly negative exponents in the monomials), where $X_d$ is shorthand for $1/X_1 \cdots X_{d-1}$. Note that the term
$$X_1^{A'} + \ldots + X_{d-1}^{A'}$$
is coprime with all the other terms of the product, and, as before, this is not a $D$th power of a polynomial.

\section{Concluding the proof}
\label{sec:conclude}

We aim to prove that almost all primes of $K(\zeta_h)$ split in at least one of the fields $K(\zeta_h, b_n^{1/h})$. By the Chebotarev density theorem it suffices to construct a subsequence $b_{n_1}, b_{n_2}, \ldots$ of $(b_n)$ such that the fields
$$K(\zeta_h, b_{n_1}^{1/h})/K(\zeta_h), K(\zeta_h, b_{n_2}^{1/h})/K(\zeta_h), \ldots$$
are linearly disjoint

In Section \ref{sec:perfect} we checked that we may apply Lemma \ref{lem:linDis} to the norm sequence $(c_n)$. Let $c_{n_1}, c_{n_2}, \ldots$ denote a subsequence given by the lemma with the base field $F(\zeta_h)$, so $$F(\zeta_h, c_{n_i}^{1/h})/F(\zeta_h), i = 1, 2, \ldots$$
are of degree $h$ and linearly disjoint. We claim that this implies that
$$F(\zeta_h, b_{n_i}^{1/h})/F(\zeta_h), i = 1, 2, \ldots$$
are of degree $h$ and linearly disjoint, too. 

By Lemma \ref{lem:garrett} it suffices to show that 
$$b_{n_1}^{e_1/h} \cdots b_{n_k}^{e_k/h} \in F(\zeta_h), 0 \le e_i < h$$
implies $e_i = 0$ for all $i$. But if $b_{n_1}^e \cdots b_{n_k}^{e_k}$ is an $h$th power in $F(\zeta_h)$, the norm $c_{n_1}^e \cdots c_{n_k}^{e_k}$ is a $h$th power in $F(\zeta_h)$, too. By the choice of $c_{n_i}$ this happens only if $e_i = 0$ for all $i$. 

Finally, note that linear disjointness over $F(\zeta_h)$ implies linear disjointness over $K(\zeta_h)$, so $K(\zeta_h, b_{n_1}^{1/h})/K(\zeta_h), K(\zeta_h, b_{n_2}^{1/h})/K(\zeta_h), \ldots$ are linearly disjoint, as desired.

\section{Discussion}
\label{sec:discussion}

The presented proof considers the primes $p$ such that $P$ has factorization type $(1, d-1)$ modulo $p$. Naturally one wonders whether other factorization types could be handled as well. 

Unfortunately, this is the only case our argument is able to handle (for irreducible $P$): The proof relies crucially on transforming the problem of zeros of linear recurrences to perfect power values of linear recurrences modulo primes. To do so, we first extract norms of the roots of $P$ via the substitution $n = k(p^{d-1} - 1)/(p-1) + r$. One has much less control over the norms in finite fields if $P$ has more than two irreducible factors modulo $p$. We then apply Lemma \ref{lem:artin}, which requires that $P$ has a root in $\mathbb{F}_p$. These two limitations result in having to consider the factorization type $(1, d-1)$.

By an approach also based on reduction to a polynomial equation, Roskam has settled the case where $P$ remains irreducible modulo $p$, assuming an analogue of Lemma \ref{lem:artin} holds for roots of $P$ of degree $d$ over $\mathbb{F}_{p}$ (namely that the multiplicative orders of the roots are of magnitude $p^d$ almost always). Note that given a linear recurrence $(a_n)$ with a squarefree characteristic polynomial $P$, the period of $(a_n)$ modulo $p$ is, for large enough primes $p$, equal to the least common multiple of the multiplicative orders of the roots of $P$ in extensions of $\mathbb{F}_p$. Hence Roskam's assumption is equivalent to the period of $(a_n)$ modulo $p$ being of order $p^d$ for almost all primes $p$ for which $P$ is irreducible modulo $p$.

A theorem of Niederreiter \cite[Theorem 4.1]{niederreiter} gives a stronger result under a weaker assumption: as long as the period of the sequence modulo $p$ is at least of magnitude $p^{d/2 + 1}$, the sequence does not only attain the value $0 \pmod{p}$, but the sequence is approximately equidistributed modulo $p$. (While we have managed to avoid the consideration of the period of $a_n \pmod{p}$, our approach does not yield equidistribution results or even non-trivial lower bounds for the number of values attained by $a_n$ modulo $p$.)

We hence see that analogies of Lemma \ref{lem:artin} to roots of $P$ in extensions of $\mathbb{F}_p$ are central to understanding the behavior of linear recurrences modulo primes. While it seems likely that such variants of Lemma \ref{lem:artin} hold (one can present a similar heuristic as for Artin's conjecture), our understanding is very limited. Unconditionally, we only know that given an integer $a, |a| > 1$, the order of $a$ modulo $p$ is almost always $> p^{1/2}$ \cite{erdos-murty}. Under GRH one  has Lemma \ref{lem:artin}, and an involved variant of Hooley's classical (conditional) solution of Artin's conjecture yields an analogue of Lemma \ref{lem:artin} in the case where $P$ is of degree $2$ and remains irreducible modulo $p$ \cite{roskamArtin}. It seems that all other cases are open, and as explained in \cite{roskamArtin}, Hooley's argument does not adapt to higher degrees without new ideas.

We conclude by remarking that the case where $P$ splits into $d$ linear factors modulo $p$ seems to be the most difficult to analyze. In these cases the period of the linear recurrence modulo $p$ divides $p-1$, and heuristically there is a positive density of split primes which are not prime divisors of the sequence (see \cite{roskamPD}). For example, we are not able to say essentially anything about the prime divisors of $3^n + 2^n + 1$ other than that there are infinitely many of them. In contrast, heuristics suggest that if $P$ is irreducible and $\deg(P) \ge 3$, then almost all non-split primes are prime divisors of the corresponding sequence (excluding degenerate cases). This suggests that the lower bound $1/(d-1)$ in Theorem \ref{thm:main} could be replaced with $1 - 1/d!$.

\bibliography{primeDivisorsLinearRecurrences}
\bibliographystyle{plain}

\end{document}